\documentclass{cpres}
\usepackage[utf8]{inputenc}
\usepackage[
backend=biber, 
style=numeric, 
sorting=nyt
]{biblatex}
\newcommand{\mi}{^*}

\title{The Jones polynomial for a torus knot with twists}
\author{Brandon Bavier, Brandy Doleshal}

\begin{document}

\maketitle

\begin{abstract}
We compute the Jones polynomial for a three-parameter family of links, the twisted torus links of the form $T((p,q),(2,s))$ where $p$ and $q$ are coprime and $s$ is nonzero. When $s = 2n$, these links are the twisted torus knots $T(p,q,2,n)$. We show that for $T(p,q,2,n)$, the Jones polynomial is trivial if and only if the knot is trivial.   
\end{abstract}

\section{Introduction}

Twisted torus knots were first introduced by Dean \cite{DeanSFSSurgeries} to study knots with Seifert fibered space surgeries. In addition to their exceptional Dehn surgeries, others have studied their knot types \cite{LeeTTKUnknots} \cite{LeeTTKTypes}, their bridge spectra \cite{BowmanBridgeSpectra}, and their ribbon length \cite{KimRibbonLength}, among other attributes. Still little is known about the various knot polynomials of the twisted torus knots. 

The Alexander polynomial has been computed when $r=2$ \cite{MortonAlexPoly}, and the knot Floer homology has been computed when $q \equiv \pm 1 \pmod p$ \cite{VafaeeKnotFloer}. No computations of the Jones polynomial appear in the literature. Here, for convenience of computations, we work with a generalization of twisted torus knots called twisted torus links, introduced by Birman and Kofman \cite{BirmanTTLinks}, to find the the Jones polynomial of the class of twisted torus knots $T(p,q, 2, n)$.

\subsection{Outline}

In Section \ref{sec:bg}, we provide the necessary background.

We prove our main results in Section \ref{sec:res}. Using the bracket polynomial and the skein relation below (in particular Equation \ref{eqn:skmain}), we reduce the bracket of a twisted torus link into a bracket of a twisted torus link with one fewer crossing, and a bracket of some other link. We will then show that this other link is, in fact, a certain torus knot, based on $p$ and $q$, with additional nugatory crossings. Then, using induction on a certain type of twist, we will get a closed form for the bracket polynomial of $T((p,q), (2,s))$. Finally, we will use the relationship between the bracket polynomial and Jones polynomial to obtain our result.

Then, in Section \ref{sec:apps}, we re-derive results of Lee by examining the auxiliary polynomial, determining conditions on $p$, $q$, and $s$ that allow the auxiliary polynomial to be trivial. In doing so, we show that for a three-parameter family of knots, the Jones polynomial is trivial if and only if the knot is trivial.

\section{Background}\label{sec:bg}

A twisted torus knot $T(p,q,r,n)$ is obtained by twisting $r$ adjacent strands of a $(p,q)$-torus knot, where $\gcd(p,q)=1$, with $n$ full twists. Notice that $T(p,q,r,n)$ can be represented by the braid $(\sigma_1 \sigma_2 \cdots \sigma_{p-1})^q (\sigma_1 \sigma_2 \cdots \sigma_{r-1})^{nr}$. A twisted torus link is a generalization of the idea of twisting strands of a torus knot, where we twist a decreasing number of strands in each section of twists. The twisted torus link is denoted $T((p_1, q_1), (p_2, q_2), \ldots ,(p_l, q_l))$, where $p_i<p_{i-1}$ for $2\leq i \leq l$ and $p_l\ge 2$, and can be represented by the braid $$(\sigma_1 \sigma_2 \cdots \sigma_{p_1-1})^{q_1}(\sigma_1 \sigma_2 \cdots \sigma_{p_2-1})^{q_2}\cdots(\sigma_1 \sigma_2 \cdots \sigma_{p_l-1})^{q_l}.$$ 

In this work, we restrict to the case when $l=2$ and $r=2$, considering the twisted torus link $T((p,q),(2,s))$ with $p\ge 3$ and $q \ge 1$. Notice here, that when $s = 2n$ for some integer $n$, this is the twisted torus knot denoted by $T(p,q,2,n)$.

As a note before the proofs, our diagrams yield the mirror image of the usual torus links and twisted torus links. However, the bracket polynomial and Jones polynomial of a link can be obtained from the bracket polynomial and Jones polynomial of the mirror image of a link by negating all the exponents (so $A^4$ becomes $A^{-4}$ and vice versa). For clarity, we use $T\mi((p,q),(2,n))$ to denote the mirror image of $T((p,q),(2,n))$.

%We can move this equation around, but it's here currently for reference
The Jones polynomial is a common knot invariant which assigns each knot a polynomial in the variable $t^{1/2}$. While there are several ways to calculate the Jones polynomial, in this paper we will use the Kauffman bracket definition. 

Given a link $L$, the Kauffman bracket polynomial, $\bracket{L}$ is defined by the following three rules:
\begin{subequations}
    \label{eqn:skein}
    \begin{align}
    \label{eqn:skcirc} \bracket{ \KPcircle} &= 1\\
    \label{eqn:skjoin} \bracket{ L\sqcup\KPcircle} &= (-A^2-A^{-2})\bracket{ L}\\
    \label{eqn:skmain} \bracket{\KPposcrossing} &= A\bracket{\KPvertical} + A^{-1}\bracket{\KPhorizontal}
    \end{align}
\end{subequations}
where $\KPcircle$ is the usual unknot, and, if $\KPposcrossing$ represents a diagram for $L$ with positive crossing $c$, $\KPvertical$ and $\KPhorizontal$ represent the diagrams where we replace $c$ with the appropriate resolution. The resolution $\KPvertical$ is called the $\infty$-resolution and $\KPhorizontal$ is called the $0$-resolution. 

While the bracket polynomial is not a knot invariant (in particular, it changes under Type 1 Reidemeister moves), it can be normalized into the auxiliary polynomial using the relation
\begin{equation}
    \label{eqn:auxpoly}
    X(L) = (-A)^{-3w(L)}\bracket{L},
\end{equation}
where $w(L)$ is the writhe of the link $L$. The auxiliary polynomial, then, is a knot invariant. From here, the Jones polynomial, $V(L)$, is obtained from $X(L)$ by replacing the $A$ with $t^{-1/4}$. 

While there are different links that share the same Jones polynomial, it is currently unknown if the Jones polynomial detects the unknot. 

\section{Results}\label{sec:res}

In this section, we provide the skein relation computations that produce the bracket polynomial for the twisted torus links $T^*((p,q),(2,s))$, where $p$ and $q$ are positive coprime integers, and $s$ is any nonzero integer. When $s = 2n$ for some nonzero integer $n$, we use this to compute the auxiliary polynomial for $T^*(p,q,2,n)$ and the Jones polynomial for the twisted torus knots $T(p,q,2,n)$, where $p$ and $q$ are coprime integers with $p\ge 3$, $q\ge 1$.

\begin{lem}
\label{lem:skeinone}
Let $p$, $q$ and $s$ be integers. Suppose $p\ge3$, $q \ge 1$, and $s\neq 0$. If $s<0$,
\[\bracket{ T^*((p,q),(2,s))} = A\bracket{ T^*((p,q),(2,s+1))} + (-1)^{3-3\abs{s}}A^{2-3\abs{s}}\bracket{ K}.\]
Otherwise
\[\bracket{T^*((p,q),(2,s))} =  A^{-1}\bracket{T\mi((p,q),(2,s-1))} + (-1)^{3\abs{s}-3} A^{3\abs{s}-2}\bracket{K}\]
where $K$ is some link.
\end{lem}

\begin{proof}
We'll start with the case where $s<0$, with a diagram for $T\mi((p,q),(2,s))$ defined by the braid 
\[(\sigma_1^{-1} \cdots \sigma^{-1}_{p-1})^q \sigma_{1}^{\abs{s}}\]

% \begin{figure}
%     \centering
%     \includegraphics[scale=1]{TTL612sresolved.png}
%     \caption{The Twisted Torus Link $T\mi((6,1),(2,s))$ with a 0-resolution}
%     \label{fig:TTL612sresolved} 
% \end{figure}

% \begin{figure}
%     \centering
%     \includegraphics[scale=1]{TTLwrithepickup.png}
%     \caption{Picking up writhe as $\ell$ increases}
%     \label{fig:ADDREF1} 
% \end{figure}

% \begin{figure}
%     \centering
%     \includegraphics[scale=1]{WritheandR2s.png}
%     \caption{Case 3}
%     \label{fig:case3} 
% \end{figure}

% \begin{figure}
%     \centering
%     \includegraphics[scale=1]{stackedknotswithshadeddisk.png}
%     \caption{$T\mi(5,2)$ stacked with shaded disk directions for removing the arc}
%     \label{fig:ADDREF3} 
% \end{figure}

To get our result, we'll use the skein relation (Equation \ref{eqn:skmain}) by selecting the bottom most crossing in the diagram. It is clear that an $\infty$-resolution at this crossing will give us $T\mi((p,q),(2,s+1))$, as $s$ is negative, so removing a crossing will increase its value. 

On the other hand, a $0$-resolution will connect two strands. For the bottom part of this resolution, we can pull it through the bottom up to the top of the diagram. For the top part, however, we will come across a nugatory crossing. As we are working in brackets, we can remove this with a Reidemeister 1 move, at the cost of adding $(-A)^{-3}$ in front of the bracket polynomial.
%TODO: Add a picture of the nugatory crossing, and then it after R1

When we do this, we are left with a diagram with one fewer total crossings, but with an added nugatory crossing. We work our way up, untwisting these two strands, until we are left with a diagram such as in Figure \ref{fig:TTL652sresolved}, which we call $K$.

\begin{figure}
    \centering
    \includegraphics[scale=1]{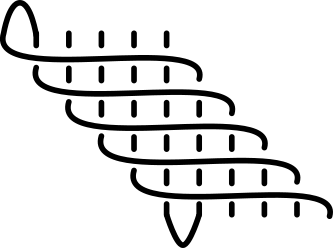}
    \caption{The Twisted Torus Link $T\mi((6,5),(2,s))$ with a 0-resolution}
    \label{fig:TTL652sresolved} 
\end{figure}

In order to get here, we added $(-A)^{-3}$ in front of the bracket each time we undid a nugatory crossing. The total number of times we do this is $\abs{s}-1$ (as we start with $\abs{s}$ crossings in this column, and then remove one with a resolution), so we add a total of $(-A)^{3-3\abs{s}}$ in front. So,
\begin{align*}
    \bracket{T\mi((p,q),(2,s))} &= A\bracket{T\mi((p,q),(2,s+1))} + A^{-1}\cdot (-A)^{3-3\abs{s}}\bracket{K}\\
        &= A\bracket{T\mi((p,q),(2,s+1))} + (-1)^{3-3\abs{s}}A^{2-3\abs{s}}\bracket{K}\\
\end{align*} 
and we are done.

On the other hand, if $s>0$, then we still use the skein relation by selecting the bottom-most crossing. Now, though, the $0$-resolution at this crossing will give us the link $T\mi((p,q),(2,s-1))$, while the $\infty$-resolution will connect two strands. Again, we work our way up with a series of Reidemeister 1 moves, but will now add $(-A)^3$ in front of the bracket for each move, with a total of $\abs{s}-1$ crossings. All combined, we'll get:
\begin{align*}
    \bracket{T\mi((p,q),(2,s))} &= A^{-1}\bracket{T\mi((p,q),(2,s-1))} + A^{1}\cdot (-A)^{3\abs{s}-3}\bracket{K}\\
        &= A^{-1}\bracket{T\mi((p,q),(2,s-1))} + (-1)^{3\abs{s}-3} A^{3\abs{s}-2}\bracket{K}
\end{align*}

\end{proof}

Note that whether $s$ is positive or negative, our diagram of $K$ will remain the same. As such, we can continue without having to worry about the sign of $s$ until the end.

\begin{lem}
\label{lem:brackets}
$K$ is isotopic to the torus link $T\mi(p-2k, q-2\ell)$, for some $k,\ell$. Further, one of the following is true:
\begin{align*}
\bracket{ K} &= A^{6\ell}\left(-A^2-A^{-2}\right)\bracket{ T\mi(p-2k, q-2\ell)}\text{, or}\\ %<- This case only applies when p and q are not coprime
\bracket{ K} &= A^{6\ell}\bracket{ T\mi(p-2k, q-2\ell)}.
%\bracket{ K} &= A^{6(\ell+1)}\bracket{ T(p-2k, q-2(\ell+1))} 
\end{align*}
\end{lem}

\begin{proof}
Our goal here is to try to get a torus knot out of $K$. As such, we'll start by isotoping the bottom half-circle arc up to the top of the diagram through the braided part.

When we do this isotopy, we are removing some number of partial twists from our diagram. In particular, as we slide the arc connecting two strands up, we remove an even number of partial twists. Outside of removing these partial twists, however, this isotopy does not remove or interact with any other crossings; we are, in effect, just removing strands from these partial twists.

\begin{figure}
    \centering
    \includegraphics[scale=1]{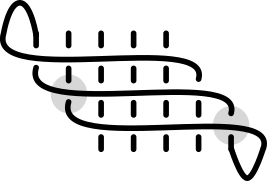}
    \caption{Each time we remove a partial twist, we pick up two nugatory crossings, shaded here.}
    \label{fig:ellwrithe} 
\end{figure}

Each time we remove two partial twists, our diagram picks up two nugatory crossings, as seen in Figure \ref{fig:ellwrithe}. We can remove these from our diagram by multiplying our bracket polynomial by $A^6$.

When we do this isotopy, our diagram will have two half-circles on the top of the diagram, in one of four configurations.
\begin{enumerate}
    \item The two half-circles connect on both sides
    \item They connect, with strands 1,2,3, as in Figure \ref{fig:TTL612sresolved}
    \item They connect, with strands 1,2,$p$, as in Figure \ref{fig:case3}
    \item They don't connect
\end{enumerate}
After this isotopy, outside of the top part of our diagram, we still have a torus link. So we need to consider what happens to the top.

% \begin{figure}
%     \centering
%     \includegraphics[scale=1]{TTL612sresolved.png}
%     \caption{After the homotopy, strands 1, 2, and 3 in $T\mi((6,1),(2,s))$ connect}
%     % \caption{The Twisted Torus Link $T\mi((6,1),(2,s))$ with a 0-resolution}
%     \label{fig:TTL612sresolved} 
% \end{figure}

\begin{figure}
    \centering
    \begin{subfigure}[b]{.35\textwidth}
        \centering
        \includegraphics[scale=1]{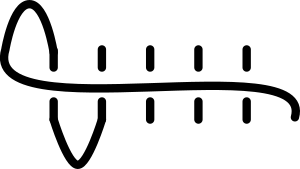}
        \caption{Case 2 Example}
        \label{fig:TTL612sresolved} 
    \end{subfigure}
    \begin{subfigure}[b]{.35\textwidth}
        \centering
        \includegraphics[scale=1]{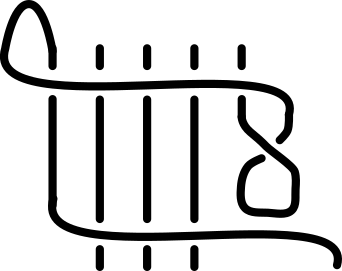}
        \caption{Case 3 Example}
        \label{fig:case3} 
    \end{subfigure}
    \label{fig:twocases}
    \caption{\textit{(A)} After the isotopy, strands 1, 2, and 3 in $T\mi((6,1),(2,s))$ connect. \textit{(B)} After the isotopy, strands 1, 2, and 5 in $T\mi((5,2),(2,s))$ connect, with some number of nugatory crossings.}
\end{figure}

In case 1, we can remove the circle using Equation \ref{eqn:skjoin}, and get back to a torus link. In case 2, the top of our diagram looks like Figure \ref{fig:TTL612sresolved}, and so quickly gives us a torus link. In case 3, we can move the strand over with a series of Reidemeister 2 moves (with no change to the bracket polynomial), and then undo the two nugatory crossings with a Reidemeister 1 move, at the cost of adding $A^6$. Note that this is effectively removing another pair of partial twists from our diagram. Finally, in case 4, we can isotope the ``bottom arc'' through the top of the diagram back to the bottom, and keep isotoping until we are in one of the other cases. When we do this, the braid representation is similar to $K$, but with two fewer strands. In fact, each time we do this isotopy, we reduce the number of strands by 2, so let $k$ be the number of times we need to slide the arc to the top.

Now, to put it all together, after completing this isotopy $k$ times, we will be in one of the first three cases mentioned above, and for some $\ell$, we will have removed a total of $2\ell$ partial twists, thus picking up $2\ell$ nugatory crossings (this includes the partial twists removed if we are in case 3). In the first case, we'll have a diagram that looks like $T\mi(p-2k, q-2\ell)\sqcup \KPcircle$ (plus $2\ell$ nugatory crossings), while in the second and third case, we'll just have $T\mi(p-2k, q-2\ell)$  (plus $2\ell$ nugatory crossings). In the first case, we can remove the circle by appending a $(-A^2-A^{-2})$ to the bracket. Then, in any of our cases, we can remove the nugatory crossings by appending a $A^{6\ell}$ to the bracket, leaving us with just $T\mi(p-2k, q-2\ell)$, and giving us our result.
\end{proof}

Now that we know that $K$ is isotopic to a torus link, we now will determine exactly what torus link it must be. In particular, we will see that the exact link is dependent on only $p$ and $q$.

\begin{lem}
\label{lem:k}
If $q = p\cdot m + \hat{q}$, let $k_p, k_1,$ and $k_2$ be such that
\[1+\hat{q}k_i\equiv i\pmod p.\]
Then $k=\min(k_p,k_1,k_2)$.
\end{lem}

\begin{proof}
    Consider the braid representation of $K$, and isotope the bottom arc up to the top, just as in Lemma \ref{lem:brackets}. Note that this takes the first strand at the bottom, and connects it to the $\hat{q}+1$ strand at the top.

    After $k$ isotopies up, we'll end up in one of three cases, as mentioned in Lemma \ref{lem:brackets}. Note that in these cases, the first strand from the bottom will connect with the first, second, or $p^{th}$ strand from the top. Each time we isotope up, we connect the $j$ strand to the $\hat{q}+j$ strand. So, for example, after one isotopy, the first strand connects to the $\hat{q}+1$ strand. After the second isotopy, it connects to the $2\hat{q}+1\pmod p$ strand, and so on. After $k$ isotopies, then, the first strand connects to the $k\hat{q}+1 \pmod p$ strand.

    We are done with our isotopies the first time the first strand from the bottom connects to the first, second, or $p^{th}$ strand from the top. This means, then, that
    \[1+\hat{q}k\equiv 1, 2, \text{ or }p\pmod p.\]
    As $k$ is the first time the bottom strand connects to the first, second, or $p^{th}$ strand, $k$ must be the first number that satisfies this equation. As such, by finding $k_p$, $k_1$, and $k_2$ that satisfy the equation $1+\hat{q}k_i\equiv i\pmod p$, we can recover $k$ by taking the minimum of these three values, and so have our result.
\end{proof}

\begin{lem}\label{lem:ell}
Given $p,q,$ and $k$ as defined above, let $\ell'$ be the solution to the equation
\[p\cdot \ell' + \epsilon = q\cdot k,\]
where $0\le \epsilon<p$. Then either $\ell=\ell'$ or $\ell=\ell'+1$.
\end{lem}

\begin{proof}

Create a covering of the disk representing the complete isotopy of Lemma \ref{lem:brackets} by stacking $k$ copies of the braid representation up, as in Figure \ref{fig:threestack}. This covering will then be the braid
\[\left(\sigma_1^{-1}\cdots\sigma_{p-1}^{-1}\right)^{q\cdot k},\]
where the first and second strand from the bottom are connected by an arc, and the first and second strand from the top are connected by an arc. 

\begin{figure}
    \centering
    \begin{subfigure}[b]{.35\textwidth}
        \centering
        \includegraphics[scale=1]{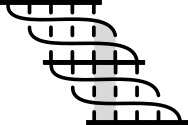}
        \caption{Two Stacks}
    \end{subfigure}
    \begin{subfigure}[b]{.35\textwidth}
        \centering    
        \includegraphics[scale=1]{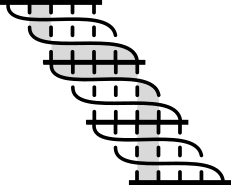}
        \caption{Three Stacks}
    \end{subfigure}
    \caption{\textit{(A)} $T\mi(5,2)$ stacked two times. The shading represents the isotopy up. \textit{(B)} $T\mi(5,2)$ stacked three times, with similar shading. Note that, as $k=2$, this second diagram will not happen, and is only for illustrative purposes.}
    \label{fig:threestack} 
\end{figure}

As we isotope our bottom arc up, every time it ``wraps around,'' we remove two twists from the diagram. Our arc will wrap around every time we have a full twist, $\left(\sigma_1^{-1}\cdots\sigma_{p-1}^{-1}\right)^{p}$. There are a total of $q\cdot k$ twists in the covering diagram; we can view this as $\ell'$ full twists, followed by some $\epsilon$ twists, with $0\le \epsilon<p$. Then it becomes clear that
\[p\cdot \ell' + \epsilon = q\cdot k.\]

If we are in either case 1 or case 2 in Lemma \ref{lem:brackets}, then $\ell=\ell'$. However, if we are in case 3, after this isotopy, strands 1,2, and $p$ will connect. This will leave us with two nugatory crossings, as can be seen in Figure \ref{fig:case3}. As such, we'll need to remove one more pair of twists, and so $\ell=\ell'+1$.

\end{proof}

\begin{cor}\label{cor:ell}
Given $p, q, k,$ and $\ell$ above, suppose $1+\hat{q}k\equiv i\pmod p$, where $i=0, 1, $ or $2$. Then
\[p\cdot \ell + i - 1 = q\cdot k.\]
\end{cor}

In particular, this corollary allows us to explicitly find $\ell$ without having to use an intermediary step of computing $\ell'$.

\begin{proof}
    Look at the equation from Lemma \ref{lem:ell}, and take both sides modulo $p$:
    \[p\cdot \ell' + \epsilon \equiv q\cdot k\pmod p.\]
    Note that $q\cdot k\equiv (p\cdot m + \hat{q})\cdot k \equiv \hat{q}\cdot k\pmod p$, so our equation simplifies to:
    \[\epsilon \equiv \hat{q}\cdot k \pmod p.\]
    By Lemma \ref{lem:k}, $\hat{q}\cdot k \equiv i-1\pmod p$, and so
    \[\epsilon\equiv i-1\pmod p.\]
    This means that there is some integer $x$ such that
    \[\epsilon = i-1 + p\cdot x.\]
    As $0\le \epsilon < p$, $x=0$ or $1$ based on $i$. We will consider each case.

    If $i=0$, then we have $0\le \epsilon = p\cdot x - 1$. This means that $x=1$, so $\epsilon = p-1$. In addition, when $i=0$, $\ell = \ell'+1$ (as mentioned in the proof above). Then, from Lemma \ref{lem:ell}, we have
    \begin{align*}
        p\cdot \ell' + \epsilon &= q\cdot k\\
        p\cdot \ell' + p-1 &= q\cdot k\\
        p\cdot (\ell'+1)-1 &= q\cdot k\\
        p\cdot \ell -1 &= q\cdot k\\
        p\cdot \ell + 0 - 1 &= q\cdot k.
    \end{align*}

    If $i=1$, then $\epsilon = p\cdot x + 1 -1 < p$, so $x=0$ and so $\epsilon=0$. Then, as $\ell = \ell'$, we get from Lemma \ref{lem:ell}:
    \begin{align*}
        p\cdot \ell + 0 &= q\cdot k\\
        p\cdot \ell + 1 - 1 &= q\cdot k
    \end{align*}
    On the other hand, if $i=2$, then $\epsilon = p\cdot x + 2 - 1 < p$, so $x=0$ and $\epsilon=1$. Similar to the previous case, $\ell = \ell'$, and so we get:
    \begin{align*}
        p\cdot \ell + 1 &= q\cdot k\\
        p\cdot \ell + 2 - 1 &= q\cdot k.
    \end{align*}
    This finishes the proof.
\end{proof}

Notice that it is possible for $\ell$ to be 0. For $p \ge 3$, this happens if and only if $q = k = 1$. 

\begin{lem}
\label{lem:recform}
Let $p$, $q$ and $s$ be integers. Let $p\ge3$, $q \ge 1$, and $s\neq 0$, and let $k$ and $\ell$ be as they are defined above. Suppose $p$ and $q$ are coprime. Then if $s<0$,
\[\bracket{ T\mi((p,q),(2,s))} = A\bracket{ T\mi((p,q),(2,s+1))} + (-1)^{1-\abs{s}} A^{2-3\abs{s}} A^{6\ell} \bracket{ T\mi(p-2k, q-2\ell)},\]
and if $s>0$,
\[\bracket{ T\mi((p,q),(2,s))} = A^{-1}\bracket{T\mi((p,q),(2,s-1))} + (-1)^{1-\abs{s}} A^{3\abs{s}-2} A^{6\ell}\bracket{ T\mi(p-2k, q-2\ell)}.\]
\end{lem}

\begin{proof}
    We will prove this for $s<0$; the proof for $s>0$ follows similarly. First, by Lemma \ref{lem:skeinone},
    \[\bracket{ T\mi((p,q),(2,s))} = A\bracket{ T\mi((p,q),(2,s+1))} + (-1)^{3-3\abs{s}} A^{2-3\abs{s}} \bracket{ K }.\]
    Note that $(-1)^{3-3\abs{s}} = (-1)^{3(1-\abs{s})}=(-1)^{1-\abs{s}}$, so we can simplify a bit:
    \[\bracket{ T\mi((p,q),(2,s)} = A\bracket{ T\mi((p,q),(2,s+1))} + (-1)^{1-\abs{s}} A^{2-3\abs{s}}\bracket{ K}.\]
    Then, by Lemma \ref{lem:brackets},
    \[\bracket{ K } = A^{6\ell}\bracket{ T\mi(p-2k, q-2\ell)},\]
    and so substituting gives us our lemma.
\end{proof}

As a remark, when $s = 0$, we can combine the two parts of Lemma \ref{lem:recform} to derive the tautology $\bracket{T\mi(p,q)} = \bracket{T\mi(p,q)}$. 

Note also that, as $\abs{s}$ decreases, $k$ and $\ell$ do not change, since they are dependent on only $p$ and $q$, as we will show in Lemmas \ref{lem:kisqhat} and \ref{lem:lisphat}. 

\begin{lem}\label{lem:TTLbracket}
    Let $p$, $q$ and $s$ be integers. Let $p\ge3$, $q \ge 1$, and $s\neq 0$, and  let $k$ and $\ell$ be as they are defined above. Suppose $p$ and $q$ are coprime. Then, if $s<0$,
    \[\bracket{T\mi((p,q),(2,s))} = A^{\abs{s}}\bracket{T\mi(p,q)} + \sum_{i=0}^{\abs{s}-1} (-1)^{1+i-\abs{s}} A^{2-3\abs{s}+4i} A^{6\ell} \bracket{T\mi(p-2k,q-2\ell)},\]
    and if $s>0$, 
    \[\bracket{T\mi((p,q),(2,s))} = A^{-\abs{s}}\bracket{T\mi(p,q)} + \sum_{i=0}^{\abs{s}-1} (-1)^{\abs{s}-1-i} A^{3\abs{s}-2-4i} A^{6\ell} \bracket{T\mi(p-2k, q-2\ell)}.\]
\end{lem}

\begin{proof}
    We'll prove this for $s<0$; the proof for $s>0$ is similar. We proceed by induction on $\abs{s}$. First, note that $T\mi((p,q),(2,0)) = T(p,q)$. If $s=-1$, then by Lemma \ref{lem:recform},
    \begin{align*}
        \bracket{ T\mi((p,q),(2,-1)} &= A\bracket{ T\mi((p,q),(2,0))} + (-1)^{1-1} A^{2-3} A^{6\ell}\bracket{ T\mi(p-2k, q-2\ell)}\\
        &= A\bracket{ T\mi(p,q)} + (-1)^{1+0-1} A^{2-3\cdot 1 + 4\cdot 0} A^{6\ell}\bracket{ T\mi(p-2k, q-2\ell)},
    \end{align*}
  so the base case is established.  Now, we suppose that
    \[\bracket{ T\mi((p,q),(2,s)) } = A^{\abs{s}}\bracket{ T\mi(p,q) } + \sum_{i=0}^{\abs{s}-1} (-1)^{1+i-\abs{s}} A^{2-3\abs{s}+4i} A^{6\ell}\bracket{T\mi(p-2k,q-2\ell)}.\]
    Look at what happens for $s-1$. First, we can use Lemma \ref{lem:recform} to get:
    \[\bracket{T\mi((p,q),(2,s-1))} = A\bracket{T\mi((p,q),(2,s))} + (-1)^{1-\abs{s-1}} A^{2-3(\abs{s-1})} A^{6\ell} \bracket{T\mi(p-2k, q-2\ell)}.\]
   From the inductive hypothesis, we have an expression for $T\mi((p,q),(2,s))$, so:
    \begin{alignat*}{2}
        \bracket{T\mi((p,q),(2,s-1))} &= A(A^{\abs{s}}\bracket{T\mi(p,q)} &&+ \sum_{i=0}^{\abs{s}-1} (-1)^{1+i-\abs{s}} A^{2-3\abs{s}+4i} A^{6\ell} \bracket{T\mi(p-2k,q-2\ell)})\\
        & &&+ (-1)^{1-\abs{s-1}} A^{2-3\abs{s-1}} A^{6\ell} \bracket{T\mi(p-2k, q-2\ell)}\\
        &= A^{\abs{s}+1}\bracket{T\mi(p,q)} &&+ \sum_{i=0}^{\abs{s}-1} (-1)^{1+i-\abs{s}} A^{2-3\abs{s}+4i+1} A^{6\ell} \bracket{T\mi(p-2k,q-2\ell)}\\
        & &&+ (-1)^{1-\abs{s-1}} A^{2-3\abs{s-1}} A^{6\ell} \bracket{T\mi(p-2k, q-2\ell)}\\
        &= A^{\abs{s}+1}\bracket{T\mi(p,q)} &&+ \sum_{i=1}^{\abs{s}} (-1)^{1+i-1-\abs{s}} A^{2-3\abs{s}+4i-4+1} A^{6\ell} \bracket{T\mi(p-2k,q-2\ell)}\\
        & &&+ (-1)^{1-\abs{s-1}} A^{2-3\abs{s-1}} A^{6\ell} \bracket{T\mi(p-2k, q-2\ell)}\\
        &= A^{\abs{s-1}}\bracket{T\mi(p,q)} &&+ \sum_{i=1}^{\abs{s-1}-1} (-1)^{1+i - \abs{s-1}} A^{2-3\abs{s-1}+4i} A^{6\ell} \bracket{T\mi(p-2k,q-2\ell)}\\
        & &&+ (-1)^{1+0-\abs{s-1}} A^{2-3\abs{s-1}+4\cdot0} A^{6\ell}\bracket{T\mi(p-2k,q-2\ell)}\\
        &= A^{\abs{s-1}}\bracket{T\mi(p,q)} &&+ \sum_{i=0}^{\abs{s-1}-1} (-1)^{1+i-\abs{s-1}} A^{2-3\abs{s-1}+4i} A^{6\ell} \bracket{T\mi(p-2k,q-2\ell)}.
    \end{alignat*}
    This completes induction, and the lemma is proved.
\end{proof}

Since the goal is to find the Jones polynomial, we prove the following lemma to provide a conversion from the bracket polynomial found in Lemma \ref{lem:TTLbracket} to the auxilliary polynomial.

\begin{lem}\label{lem:BracketToAux}
    Let $D$ be the diagram of some link, and suppose that the bracket polynomial of $D$ can be written as
    \[\bracket{D} = \sum_{i=1}^n f_i(A)\bracket{D_i},\]
    where $D_i$ are link diagrams, and $f_i$ is a polynomial in $A$. Then,
    \[X(D) = (-A)^{-3w(D)}\sum_{i=1}^n f_i(A)(-A)^{3w(D_i)}X(D_i)\]
\end{lem}

\begin{proof}
    First, recall the definition of the normalized bracket polynomial, $X(D_i)$:
    \[X(D_i) = (-A)^{-3w(D_i)}\bracket{D_i}.\]
    This means, then, that
    \[\bracket{D_i} = (-A)^{3w(D_i)} X(D_i).\]
    Then, for our specific diagram, we can write:
    \[\bracket{D} = \sum_{i=1}^n f_i(A) (-A)^{3w(D_i)}X(D_i),\]
    and so
    \begin{align*}
        X(D) &= (-A)^{-3w(D)}\bracket{D}\\
             &= (-A)^{-3w(D)}\sum_{i=1}^n f_i(A) (-A)^{3w(D_i)}X(D_i).
    \end{align*}
\end{proof}

We will also need the following lemma to obtain the auxiliary polynomial.

\begin{lem}\label{lem:relprime}
If $p$ and $q$ are relatively prime, then $p-2k$ and $q-2\ell$ are relatively prime.
\end{lem}

\begin{proof}
    First, if $p$ and $q$ are relatively prime, then note that the $i$ from Lemma \ref{lem:k} is either 0 or 2. If $i=1$, then we would have $\hat{q}k\equiv 0\pmod p$, with $k<p$, so $\hat{q}$ would have to share common factors with $p$. However, if $\hat{q}$ shares some common factor with $p$, then so too must $q=p\cdot m + \hat{q}$, a contradiction. 

    If $i \in \{0,2\}$, from Corollary \ref{cor:ell}, we have, after some rearranging,
    \[p\cdot \ell - q\cdot k = 1-i = \pm 1.\]
    Now, look at $p-2k$ and $q-2\ell$. If we can find some integers $x$ and $y$ such that
    \[x(p-2k) + y(q-2\ell) = \pm 1,\]
    then we will be done. Take $x=\ell$, and $y=-k$. Then:
    \begin{align*}
        \ell(p-2k) -k(q-2\ell) &= \ell \cdot p - 2\ell k - k\cdot q + 2\ell k\\
        &= \ell\cdot p-k\cdot q -2\ell k + 2\ell k\\
        &= \pm 1 - 0.
    \end{align*}
    And so we are done.
\end{proof}

\begin{lem}
\label{lem:AuxPolyNeg}
Let $p$, $q$ and $s$ be integers. Suppose $p\ge3$, $q \ge 1$,$s<0$, and $\gcd(p,q) = 1$. The auxiliary polynomial $X(T\mi((p,q),(2,s)))$ is 
\[\frac{(-1)^{\abs{s}} A^{2(p-1)(q-1)-2\abs{s}}}{1-A^{8}} \left[ X'_{p,q} - (-1)^{\abs{s}}\sum_{i=0}^{\abs{s}-1}(-1)^i A^{4i} A^{2-4\abs{s} + 4(k+\ell-k\ell) + 2(kq + \ell p)}X'_{p-2k,q-2\ell} \right] \]
where $X'_{p,q} = 1 - A^{4p+4} - A^{4q+4} + A^{4p+4q}$ and $X'_{p-2k,q-2\ell} = 1 - A^{4p
+4 - 8k} - A^{4q+4 - 8\ell} + A^{4p+4q - 8k - 8\ell}$.
\end{lem}

\begin{proof}
Assume the given hypotheses on $p$, $q$ and $s$. From Lemma \ref{lem:TTLbracket}, we know that the bracket polynomial of $T\mi((p,q),(2,s))$ is given by     \[\bracket{T\mi((p,q),(2,s))} = A^{\abs{s}}\bracket{T\mi(p,q)} + \sum_{i=0}^{\abs{s}-1} (-1)^{1+i-\abs{s}} A^{2-3\abs{s}+4i} A^{6\ell} \bracket{T\mi(p-2k,q-2\ell)}.\]

Rewriting the sum $\displaystyle\sum_{i=0}^{\abs{s}-1} (-1)^{1+i-\abs{s}} A^{2-3\abs{s}+4i}$ as $\displaystyle (-1)^{1-\abs{s}}\sum_{i=0}^{\abs{s}-1} (-1)^i A^{4i}A^{2-3\abs{s}}$ and using Lemma \ref{lem:BracketToAux}, we have that $X(T\mi((p,q),(2,s)))$ is of the form 
\begin{align*}
(-A)^{-3w(D)} A^{\abs{s}}\bigg[&(-A)^{3w(D_{p,q})}X_{p,q} \\
&+(-1)^{1-\abs{s}} \sum_{i-1}^{\abs{s}}(-1)^i A^{4i}A^{2- 4\abs{s} +6\ell}(-A)^{3w(D_{p-2k,q-2\ell})}X_{p-2k,q-2l}\bigg],
\end{align*}

where the notation $X_{m,n}$ refers to the auxiliary polynomial for the torus knot $T\mi(m,n)$ and $D_{m,n}$ refers to our diagram for $T\mi(m,n)$. In this context, we are using the mirror image of the usual braid closure as the diagram for the knot $T\mi((p,q),(2,s))$, so $w(D) = -(p-1)q + \abs{s}$. Similarly $w(D_{p,q}) = -(p-1)q$ and $w(D_{p-2k,q-2l}) = -(p-2k-1)(q - 2l)$. 

We write $w(D_{p-2k,q-2l})$ instead as $-(p-1)q + 2kq + 2 \ell p-2\ell - 4k \ell$. Then we can factor $A^{-3(p-1)q}$ out of each term in the sum to cancel the $A^{3(p-1)q}$ in the front of the expression, so we have $X(T\mi((p,q),(2,s)))$ in the form 
\[ (-1)^{3\abs{s}}A^{-2\abs{s}} \left[X_{p,q} + (-1)^{1-\abs{s}}\sum_{i=0}^{\abs{s}-1}(-1)^i A^{4i}A^{2-4\abs{s}+6\ell}(-A)^{ 6kq + 6 \ell p-6\ell - 12k \ell}X_{p-2k,q-2l}\right].\]
We combine exponents further to obtain 
\[ (-1)^{\abs{s}}A^{-2\abs{s}} \left[X_{p,q} - (-1)^{\abs{s}}\sum_{i=0}^{\abs{s}-1}(-1)^i A^{4i}A^{2-4\abs{s} + 6kq + 6 \ell p - 12k \ell}X_{p-2k,q-2l}\right].\]

Now we use Lemma \ref{lem:relprime} and the fact that the Jones polynomial for the torus knot $T(m,n)$ is given by $\displaystyle\frac{t^{(m-1)(n-1)/2}(1 - t^{m+1} - t^{n+1} + t^{m+n})}{1-t^2}$ \cite{JonesPolyTorusKnot}. For ease of notation, we work with the auxiliary polynomial by substituting $A^{4}$ for $t$, since we are working with the mirror image, so that 

$$\displaystyle X_{p,q} = \frac{A^{2(p-1)(q-1)}(1 - A^{4p+4} - A^{4q+4} + A^{4p+4q})}{1-A^{8}}$$
and 
$$\displaystyle X_{p-2k,q - 2l} = \frac{A^{2(p-2k-1)(q-2\ell-1)}(1 - A^{4p+4-8k} - A^{4q+4 - 8\ell} + A^{4p+4q - 8k - 8\ell})}{1-A^{8}}.$$

Using the notation $X'_{p,q} = 1 - A^{4p+4} - A^{4q+4} + A^{4p+4q}$ and $X'_{p-2k,q-2\ell} = 1 - A^{4p+4 - 8k} - A^{4q+4 - 8\ell} + A^{4p+4q - 8k - 8\ell}$, we can rewrite $X_{p,q}$ and $X_{p-2k,q-2l}$ as $\displaystyle X_{p,q} = \frac{A^{2(p-1)(q-1)}X'_{p,q}}{1-A^{8}}$ and as $\displaystyle X_{p-2k,q - 2l} = \frac{A^{2(p-1)(q-1)}A^{-4\ell p + 4 \ell -4 kq +8 k \ell + 4k}X'_{p-2k,q-2l}}{1-A^{8}}$. Combining these with the previous sum and simplifying, we have our desired result.
\end{proof}

\begin{thm}
Let $p$, $q$ and $s$ be integers. Suppose $p\ge3$, $q \ge 1$, $\gcd(p,q) = 1$,   and that $s=2n$ for some negative integer $n$. The Jones polynomial of $T((p,q),(2,s)) = T(p,q,2,n)$ is

$$\frac{ t^{\frac{(p-1)(q-1)-\abs{s}}{2}}}{1-t^{2}} \left[J'_{p,q} - \sum_{i=0}^{\abs{s}-1}(-1)^i t^{i - \abs{s} + k+\ell- k\ell + \frac{(kq + \ell p+1)}{2}}J'_{p-2k,q-2\ell} \right]$$

where $J'_{p,q} = 1-t^{p+1}-t^{q+1}+t^{p+q}$ and $J'_{p-2k,q-2\ell} = 1-t^{p+1-2k}-t^{q+1-2\ell}+t^{p+q-2k-2\ell}$.
\end{thm}

\begin{proof}
The Jones polynomial is obtained from Lemma \ref{lem:AuxPolyNeg} by replacing $A$ with $t^{1/4}$. The typical change of variables replaces $A$ with $t^{-1/4}$, but here we must account for the fact that we have been working with the mirror image.
\end{proof}

\begin{lem}\label{lem:AuxPolyPos}
Let $p$, $q$ and $s$ be integers. Suppose $p\ge3$, $q \ge 1$,$s>0$, and $\gcd(p,q) = 1$.  Then the auxiliary polynomial $X(T\mi((p,q),(2,s)))$ is
\[\frac{(-1)^{s} A^{2(p-1)(q-1)+2s}}{1-A^{8}} \left[ X'_{p,q} - (-1)^{s}\sum_{i=0}^{s-1}(-1)^i A^{-4i} A^{4s-2 + 4(k+\ell-k\ell) + 2(kq + \ell p)}X'_{p-2k,q-2\ell} \right] \]
where $X'_{p,q} = 1 - A^{4p+4} - A^{4q+4} + A^{4p+4q}$ and $X'_{p-2k,q-2\ell} = 1 - A^{4p
+4 - 8k} - A^{4q+4 - 8\ell} + A^{4p+4q - 8k - 8\ell}$.
\end{lem}

The proof of this lemma is identical in structure to the proof of Lemma \ref{lem:AuxPolyNeg}. 

\begin{thm}
Let $p$, $q$ and $s$ be integers. Suppose $p\ge3$, $q \ge 1$, $\gcd(p,q) = 1$,   and that $s=2n$ for some positive integer $n$. The Jones polynomial of $T((p,q),(2,s)) = T(p,q,2,n)$ is

$$\frac{ t^{\frac{(p-1)(q-1)+s}{2}}}{1-t^{2}} \left[J'_{p,q} - \sum_{i=0}^{s-1}(-1)^i t^{-i + s + k + \ell - k \ell + \frac{ kq+\ell p-1}{2}}J'_{p-2k,q-2\ell} \right]$$

where $J'_{p,q} = 1-t^{p+1}-t^{q+1}+t^{p+q}$ and $J'_{p-2k,q-2\ell} = 1-t^{p-2k+1}-t^{q-2\ell+1}+t^{p+q-2k - 2\ell}$.
\end{thm}

\begin{proof}
The Jones polynomial is obtained from Lemma \ref{lem:AuxPolyPos} by replacing $A$ with $t^{1/4}$. 
\end{proof}

\section{Applications}\label{sec:apps}

In this section, we use the auxiliary polynomial from the previous section to relate these computations to work of Sangyop Lee. In \cite{LeeTTKUnknots} and \cite{LeeTTKTypes}, Lee finds the twisted torus knots that are unknotted. Here we re-state Lee's theorem for the case when $r = 2$ and for the twisted torus links we consider.

\begin{thm}\label{thm:LeesUnknots}
Let $p, q, s$ be integers with $p \ge 3$, $q \ge 1$ and $s$ even and negative, and let $p$ and $q$ be relatively prime. The twisted torus link $T((p, q), (2, s))$ is the unknot if and only if $(p,q,s)$ is one of the following:
\begin{enumerate}
    \item $(3, 2, -2)$;
    \item $(2m \pm 1, 2, -2m)$ for some integer $m \ge 2$;
    \item $(n, 1, -2)$ for some integer $n\ge 3$
\end{enumerate}
\end{thm}

In addition to the unknots listed above, Lee lists as an unknot the twisted torus knot $T(2,1,2,-1)$. Further, Lee's theorem is stated so that the roles of $p$ and $q$ can be swapped. In our current setting, it does not make sense to allow $p\le2$ because the definition of a twisted torus link $T((p_1, q_1),(p_2,q_2),\ldots,(p_l,q_l))$ requires that $p_{i} < p_{i-1}$ for $2 \le i \le l$. As such, twisted torus knots with $p\le 2$ do not have a representation as a twisted torus link. While our work does not apply to these knots, it is easy to see that the twisted torus knot $T(2,q,2,n)$ is isotopic to the torus knot $T(2,q+2n)$. Then $T(2,1,2,-1)$, $T(2, 3, 2, -1)$, and $T(2, 2m\pm 1, 2, -m)$ are all isotopic to one of the torus knots $T(2,\pm 1)$, which are both the unknot. Therefore, each of these twisted torus knots has a trivial Jones polynomial. 

\begin{prop} \label{prop:LeeisRight}
Let $p$, $q$ and $s$ be integers. Suppose $p\ge3$, $q \ge 1$, $\gcd(p,q) = 1$, and that $s$ is even and negative. The Jones polynomial of the twisted torus link $T((p,q),(2,s))$ is trivial if and only if $T((p,q),(2,s))$ is the unknot.
\end{prop}

Here we call $T((p,q),(2,s))$ a link to distinguish the phrases ``twisted torus link" and ``twisted torus knot," but it is actually a knot when $p$ and $q$ are coprime and $s$ is even.

Before we begin the proof of Proposition \ref{prop:LeeisRight}, we prove two lemmas that relate $k$ and $\ell$ to the congruence classes of $q\pmod p$ and $p \pmod q$, respectively. We will need this for the proof, but additionally, this information could be useful for applying this work to other circumstances.

\begin{lem}\label{lem:kisqhat}
When $p$ and $q$ are relatively prime, $k = \min(\hat{q}^{-1}, p - \hat{q}^{-1})$, where $q \equiv \hat{q} \mod p$ and $1 \le \hat{q} \le p-1 $.
\end{lem}

\begin{proof}
When $p$ and $q$ are relatively prime, as $k<p$ and $\ell<q$, we must have $p\ell - qk \neq 0$. This means, then that the $i$ in Corollary \ref{cor:ell} is either 0 or 2, and we have, after some rearranging,
\[q\cdot k - p\cdot\ell= \pm 1.\]
Now, taking things modulo $p$, we get:
\[\hat{q}\cdot k \equiv \pm 1 \pmod p \implies k \equiv \pm \hat{q}^{-1}\pmod p. \]
As $0\le k<p$, we have exactly two options. If $i=2$ and $k\equiv \hat{q}^{-1}\pmod p$, then $k=\hat{q}^{-1}$. Otherwise, $i=0$ and $k\equiv -\hat{q}^{-1}\pmod p$, and so $k=p-\hat{q}^{-1}$. As $k$ is defined to be the first solution to the equation in Lemma \ref{lem:k}, $k$ must be the smaller of these two values, giving us our result.
%When $p$ and $q$ are relatively prime, the $i$ in Lemma \ref{lem:k} is either 2 or $p$. Since $\gcd(p, q) = \gcd(p, \hat{q}) = 1$, we can solve for $k$ in the given equation, obtaining $k \equiv \pm \hat{q}^{-1} \pmod p$. Since $1 \le k \le p-1$, we have $k =  \hat{q}^{-1}$ or $p - \hat{q}^{-1}$. Since $k$ is defined to be the first instance that provides a solution to the equation in Lemma \ref{lem:k}, $k$ must be the smaller of the two values.
\end{proof}

Note that one consequence of this lemma is that $2k < p$.

\begin{lem}\label{lem:lisphat}
When $p$ and $q$ are relatively prime, $\ell = \hat{p}^{-1}$ or $\ell = q- \hat{p}^{-1}$, where $p \equiv \hat{p} \pmod q$ and $1\le \hat{p} \le q-1$.
\end{lem}

\begin{proof}
As above, when $p$ and $q$ are relatively prime, the $i$ from Corollary \ref{cor:ell} is either 0 or 2. As such, we get:
\[p\cdot \ell \pm 1 = q\cdot k.\]
Rearranging, and taking both sides modulo $q$, we get
\[p\cdot \ell \equiv \pm 1 \pmod{q}.\]
Let $\hat{p} = p\pmod q$. Then, using the same proof as Lemma \ref{lem:kisqhat}, and the fact that $1\le \ell< q$, we get our result.

%From Lemma \ref{lem:k} and Lemma \ref{lem:kisqhat}, when $i = 2$, $k = \hat{q}^{-1}$, so the equation from Lemma \ref{lem:ell} becomes $p\cdot \ell' + \epsilon = \hat{q}^{-1}q$. This equation tells us that $\epsilon  \equiv 1 \pmod p$, and since $1 \le \epsilon \le p-1$, we have $\epsilon = 1$. Further, since we are in the case that $\ell = \ell'$ in Lemma \ref{lem:ell}, we have the equation $p\cdot \ell + 1 = kq$. Now we can see that $\hat{p}\cdot \ell + 1 \equiv 0 \pmod q$, where $\hat{p}$ is the equivalence class of $p$ modulo $q$. Since $\hat{p}$ and $q$ are relatively prime, we can solve for $\ell$ to obtain $\ell \equiv -\hat{p}^{-1} \pmod q$. By its definition, $1\le \ell \le q$, so we have that $\ell = q - \hat{p}^{-1}$

%In the case where the $i$ from Lemma \ref{lem:k} is $p$, we have that $k = p-\hat{q}^{-1}$, so $\epsilon = p-1$. Here, we are in the case that $\ell = \ell' + 1$, so the equation from Lemma \ref{lem:ell} becomes $p \cdot (\ell -1) + p-1 = kq$. The left side of this equation simplifies to obtain  $p \cdot \ell -1 = kq$, so that we see $\hat{p}\cdot \ell - 1\equiv 0 \pmod q$. Again, we solve for $\ell$ in this congruence and note that $1 \le \ell \le q$ to determine that $\ell = \hat{p}^{-1}$.
\end{proof}

The following lemma is used several times in the proof of Proposition \ref{prop:LeeisRight}.

\begin{lem}\label{lem:smallcancellation}
Let $p$, $q$, $s$ be integers with $p\ge3$, $q\ge 1$, $\gcd(p,q) = 1$, and $s$ even and negative. When $X'_{p-2k,q-2\ell} = 1-A^8$, the auxiliary polynomial is \[\frac{ A^{2(p-1)(q-1)-2\abs{s}}}{1-A^{8}} \left[ 1 - A^{4p+4} - A^{4q+4} + A^{4p+4q} - A^{\sigma} + A^{\sigma + 4} + A^{\sigma + 4\abs{s}} - A^{\sigma + 4\abs{s}+4} \right], \] where $\sigma = 2 - 4\abs{s} + 4(k + \ell - k \ell) + 2 (kq + \ell p )$.
\end{lem}

\begin{proof}
Assuming the given hypotheses, the sum portion of the auxiliary polynomial from Lemma \ref{lem:AuxPolyNeg} is $\displaystyle \sum_{i=0}^{\abs{s}-1}(-1)^i A^{4i} A^{\sigma}(1-A^8)$ with $\sigma = 2-4\abs{s} + 4(k+\ell-k\ell) + 2(kq + \ell p)$. 

Because the exponent of each term of the sum is 4 greater than the previous exponent, the $j$ term of $\displaystyle \sum_{i=0}^{\abs{s}-1}(-1)^i A^{4i} A^{\sigma}$ cancels with the $j-2$ term of $\displaystyle \sum_{i=0}^{\abs{s}-1}(-1)^i A^{4i} A^{\sigma}(- A^8)$, for $2\le j\le \abs{s}-1$. This leaves the $i=0$ and $i=1$ terms of $\displaystyle \sum_{i=0}^{\abs{s}-1}(-1)^i A^{4i} A^{\sigma}$ and the $i=\abs{s}-2$ and $i = \abs{s}-1$ terms of $\displaystyle \sum_{i=0}^{\abs{s}-1}(-1)^i A^{4i} A^{\sigma}(- A^8)$.

Using the hypothesis that $s$ is even, the sum becomes $A^{\sigma} - A^{\sigma + 4} - A^{\sigma + 4\abs{s} - 8 + 8} + A^{\sigma + 4 \abs{s} -4 + 8}$. Simplifying this expression and substituting it for the equivalent expression in the auxiliary polynomial, we arrive at the result.
\end{proof}

\begin{proof}[Proof of Proposition \ref{prop:LeeisRight}]
First, note that the Jones polynomial equals 1 if and only if the auxiliary polynomial equals 1. Here we work with the auxiliary polynomial for ease of notation. Further, for $K$ a knot with mirror image $K\mi$, the auxiliary polynomial $X(K)$ is obtained from $X(K\mi)$ by negating the exponents. Hence $X(K\mi) = 1$ if and only if $X(K) = 1$, so we can use the formulas derived in the previous section.

By plugging the parameters from Theorem \ref{thm:LeesUnknots} into the formula for the auxilliary polynomial from Lemma \ref{lem:AuxPolyNeg}, we see that each of these knots have trivial auxilliary polynomial, and thus trivial Jones polynomial. Hence, a twisted torus knot having these parameters is sufficient for having a trivial Jones polynomial. We proceed with proving that $T((p,q),(2,s))$ having $(p,q,s)$ matching these parameters is necessary for having trivial Jones polynomial.

If we are looking for a trivial auxiliary polynomial, we should expect the polynomial to have a single term. As such, with some rearranging of our equation, we will get, for some exponent $\alpha$:
\[1-A^8 =  A^{2(p-1)(q-1)-2\abs{s}} \left[X'_{p,q} - \sum_{i=0}^{\abs{s}-1}(-1)^i A^{4i+\alpha} X'_{p-2k,q-2\ell}\right].\]
Thus, to prove the auxiliary polynomial is nontrivial, it is enough to show that, after any potential cancellations, $X'_{p,q} - \sum (-1)^i A^{4i+\alpha} X'_{p-2k,q-2\ell}$ has at least three terms.

While it is tempting to analyze all different situations carefully, we will take a bit cruder, but more effective, approach. As $X'_{p,q}$ has four terms, if we can show that the second part, the summation, has at least seven terms, we will be done.

Focus on just the summation coming from $T\mi(p-2k, q-2\ell)$:
\[ \left(\sum_{i=0}^{\abs{s}-1} (-1)^i A^{4i+\alpha}\right)\left(1-A^{4p+4-8k}-A^{4q+4-8\ell}+A^{4p+4q-8k-8\ell}\right).\]
We can view this sum as four different ``degree shifts'' of the same polynomial, namely $\sum (-1)^i A^{4i+\alpha}$. 

For now, suppose $p-2k<q-2\ell$, and write out the polynomials coming from $1$ (a degree shift of 0) and $A^{4p+4-8k}$ (a degree shift of $4p+4-8k$):
\begin{center}
\begin{tabular}{c || c c c c c c }
    $1$                 & $A^{\alpha}$ & $-A^{\alpha+4}$ & $\cdots$ & $(-1)^{p+1-2k} A^{\alpha+4(p+1-2k)}$ & $(-1)^{p+1-2k+1} A^{\alpha+4(p+1-2k+1)}$ & $\cdots$\\
    $A^{4p+4-8k}$       & & & & $-A^{\alpha+4p+4-8k}$ & $+A^{\alpha+4p+4-8k+4}$ & $\cdots$ \\
\end{tabular}
\end{center}
While we can write the other two polynomials in a similar fashion, the first terms will all be ``above'' the first term of $A^{4p+4-8k}$. As such, the polynomial coming from $1$ will have several terms free until it starts to interact with the polynomial coming from $A^{4p+4-8k}$. In particular, the first time these two polynomials can interact is when $i=p-2k+1$ in the $1$ polynomial and $i=0$ in the $A^{4p+4-8k}$ polynomial. As such, there are at least $p-2k+1$ uncancelled terms (with the $+1$ coming from the fact that we start at $0$ in the sum) in the $1$ polynomial.

On the other hand, look at the polynomials coming from $A^{4q+4-8\ell}$ and $A^{4p+4q-8k-8\ell}$. In particular, we are looking for the \textit{last} time the two can interact. For the $A^{4q+4-8\ell}$ polynomial, this will be when $i=\abs{s}-1$, with exponent $4q+4-8\ell+4\abs{s}-4$. For the $A^{4p+4q-8k-8\ell}$ polynomial, then, we will have 
\begin{align*}
    4p+4q-8k-8\ell + 4i &= 4q+4-8\ell+4\abs{s}-4\\
    4i &= 4\abs{s}-4-4p+8k\\
     i &=  \abs{s}-1-p+2k.
\end{align*}
This means, then, that the $A^{4p+4q-8k-8\ell}$ polynomial has at least $p-2k$ terms that don't interact with any other polynomial.

Putting this all together, we will get at least $2(p-2k)+1$ uncancelled terms in this polynomial. If $p-2k\ge 3$, then we must be done, as $2\cdot 3+1 = 7$, the minimal number of terms we need in the sum to have nontrivial polynomial. If $p-2k=2$, then we have at least $2\cdot 2+1 = 5$ terms, not quite enough. However, we can find a couple more uncancelled terms here. 

Note that $(-1)^{p+1-2k} = -1$. As such, the $i=p+1-2k$ term in the $1$ polynomial does not cancel with the $i=0$ term in the $A^{4p+4-8k}$ polynomial. In fact, as the signs match up, none of the terms in these two polynomials will cancel! The first time, then, we could have cancelled terms is either after we reach the end of the $1$ polynomial (and so have $\abs{s}$ terms), or when both the $A^{4q+4-8\ell}$ and the $A^{4p+4q-8k-8\ell}$ polynomials are interacting. As $p-2k=2$, the first time the $A^{4p+4q-8k-8\ell}$ polynomial can interact with the others, the exponent is $4q-8\ell+8$. Then the earliest term in the $1$ polynomial that can cancel is when $4i=4q-8\ell+8$, and so $i=q-2\ell+2$. As $2=p-2k<q-2\ell$, this means $i$ is at least $5$, and so there are at least 5 terms in the $1$ polynomial that don't cancel. Combined with the 2 terms in the $A^{4p+4q-8k-8\ell}$ polynomial, this gives us the required 7 terms to be nontrivial.
%Thus, by Lemma \ref{lem:AuxPolyNeg},
%\[1 = \frac{(-1)^{\abs{s}}A^{-2\abs{s}} A^{2(p-1)(q-1)}}{1-A^{8}} \left[ X'_{p,q} - (-1)^{\abs{s}}\sum_{i=0}^{\abs{s}-1}(-1)^i A^{4i} A^{2-4\abs{s} + 4(k+\ell-k\ell)) + 2(kq + \ell p)}X'_{p-2k,q-2\ell} \right]. \]

Since the roles of $p-2k$ and $q - 2\ell$ can be reversed in the above argument, we know that at least one of $p-2k$ and $q-2\ell$ is strictly smaller than 2. Since $p \ge 3$ and $2k < p$, we know that $p-2k \ge 1$. We have the following five cases to consider:
\begin{enumerate}
\item\label{case:qis1} $q=1$ and $\ell = 0$,
\item\label{case:qis2} $q -2\ell = 0$,
\item\label{case:p1q1} $p-2k = q-2\ell = 1$,
\item\label{case:q1pbig} $1 = q - 2\ell < p - 2k$, and
\item\label{case:p1qbig} $1 = p-2k < q - 2\ell$.
\end{enumerate}

In cases \ref{case:p1q1}, \ref{case:q1pbig}, and \ref{case:p1qbig}, we may assume $q \ge 3$.

\textit{Case \ref{case:qis1}}: We begin by considering the case where $q = 1$. We have $k=1$ and $\ell=0$, so the polynomial from Lemma \ref{lem:AuxPolyNeg} is $\displaystyle \frac{A^{-2\abs{s}}}{1-A^{8}} \left[1-A^8 - \sum_{i=0}^{\abs{s}-1}(-1)^i A^{4i} A^{8-4\abs{s}} (1-A^8) \right]$, which is equal to $\displaystyle A^{-2\abs{s}}\left(1 - \sum_{i=0}^{\abs{s}-1}(-1)^i A^{4i} A^{8-4\abs{s}}\right) = A^{-2\abs{s}} - \sum_{i=0}^{\abs{s}-1}(-1)^i A^{4i} A^{8-6\abs{s}} $. In order for this polynomial to be 1, some term must have exponent 0. Then it must be the case that $8-6\abs{s} + 4i = 0$ for some $0 \le i \le \abs{s}-1$. However this would mean that $4i = 6\abs{s} - 8$. Since we know $0 \le 4i \le 4\abs{s} - 4$, we have that $4i \le 6\abs{s} -8$ with equality if and only if $\abs{s} = 2$. When $\abs{s} = 2$, the auxiliary polynomial is trivial. Notice that we have provided no restrictions on $p$ here, and so this corresponds to the family of twisted torus links of the form $T((p,1),(2,-2))$, which is the family given in Theorem \ref{thm:LeesUnknots}\textit{(3)}.

\textit{Case \ref{case:qis2}}: Next we assume that $q = 2\ell$. As a consequence of Lemma $\ref{lem:lisphat}$, $q$ and $\ell$ are coprime, so $\ell=1$ and $q=2$. The equation from Corollary \ref{cor:ell} tells us that $p = 2k -i + 1$, where $i \in \{0,2\}$. That is, $p = 2k \pm 1$. 

First consider $p = 2k + 1$. Using Lemma \ref{lem:smallcancellation}, with $\sigma = 8 - 4\abs{s} + 8k$, we write the auxiliary polynomial as
\[ \frac{A^{4k-2\abs{s}}}{1-A^{8}} \left[ 1-A^{8k+8} - A^{12} + A^{8k+12} -  A^{8-4\abs{s} + 8k} + A^{12 - 4\abs{s} + 8k} + A^{8k+8} - A^{8k+12} \right] \]

After further cancellation, we have $\displaystyle \frac{A^{4k-2\abs{s}}}{1-A^{8}} \left[ 1 - A^{12}  -  A^{8-4\abs{s} + 8k} + A^{12 - 4\abs{s} + 8k} \right]$. In order for this polynomial to be trivial, it must be that the exponent of $A$ outside the brackets is zero or the $A^{0}=1$ term inside the brackets cancels with another term. Only $A^{8-4\abs{s}+8k}$ can cancel with the 1. That is,  $4k - 2\abs{s} = 0$ or $8-4\abs{s} + 8k = 0$, so that $\abs{s}$ is either $2k$ or $2(k+1)$. From Theorem \ref{thm:LeesUnknots}\textit{(2)}, the former corresponds to $(p,q,s) = (2m+1, 2, -2m)$, where $m = k$ and the latter corresponds to $(p,q,s) = (2m - 1, 2, -2m)$, where $m = k+1$. Notice that when $m = k = 1$, we have case \textit{(1)} from Theorem \ref{thm:LeesUnknots}.

Next consider $p = 2k-1$. The auxiliary polynomial is \[\frac{ A^{2(2k-2)-2\abs{s}}}{1-A^{8}} \left[ 1-A^{8k} - A^{12}+A^{8k+4} - \sum_{i=0}^{\abs{s}-1}(-1)^i A^{4i} A^{4-4\abs{s}  + 8k}(-A^{4}+A^{-4}) \right] .\]
We factor an $A^{-4}$ from $(-A^{4}+A^{-4})$ to obtain:
\[\frac{ A^{2(2k-2)-2\abs{s}}}{1-A^{8}} \left[ 1-A^{8k} - A^{12}+A^{8k+4} - \sum_{i=0}^{\abs{s}-1}(-1)^i A^{4i} A^{-4\abs{s}  + 8k}(1-A^8) \right] .\]

As in Lemma \ref{lem:smallcancellation}, the polynomial becomes

\[\frac{ A^{2(2k-2)-2\abs{s}}}{1-A^{8}} \left[ 1-A^{8k} - A^{12}+A^{8k+4} - A^{-4\abs{s} + 8k} + A^{4-4\abs{s}+ 8k} + A^{8k} -A^{8k+4} \right]. \]

After combining like terms, we have $\displaystyle \frac{ A^{2(2k-2)-2\abs{s}}}{1-A^{8}} \left[ 1 - A^{12} - A^{-4\abs{s} + 8k} + A^{4-4\abs{s}+ 8k} \right]$.
In order for this polynomial to be 1, either $2(2k-2) -2\abs{s} = 0$ or $-4|s| + 8k = 0$. Then we have that either $\abs{s}= 2(k-1)$ or $\abs{s} = 2k$.   These correspond to case \textit{(2)} of Theorem \ref{thm:LeesUnknots} (or case \textit{(1)} if $\abs{s} = k = 2$). 

Now we have completely analyzed the cases where $q = 1$ and $q = 2$. For the remaining cases, we assume $q$ is at least 3.

\textit{Case \ref{case:p1q1}}: Next, we consider the case that $p-2k = q - 2\ell = 1$. Then $p = 2k+1$ and $q= 2 \ell + 1$, where each of $k$ and $\ell$ is at least 1. Using Lemma \ref{lem:smallcancellation}, the auxiliary polynomial becomes \[\frac{A^{8k\ell - 2 \abs{s}}}{1 - A^8}\left[1 - A^{8k+8} - A^{8\ell + 8} + A^{8k + 8 \ell + 8} - A^{\sigma} + A^{\sigma + 4} + A^{\sigma + 4 \abs{s}} - A^{\sigma+ 4\abs{s} + 4}\right],\] where $\sigma = 2 - 4 \abs{s} + 6k + 6 \ell + 4 k \ell$. This polynomial is equal to 1 if and only if \[1 - A^{8k+8} - A^{8\ell + 8} + A^{8k + 8 \ell + 8} - A^{\sigma} + A^{\sigma + 4} + A^{\sigma + 4 \abs{s}} - A^{\sigma+ 4\abs{s} + 4} = A^{2\abs{s}-8 k \ell}(1-A^8).\]

Because $k,\ell \ge 1$, this can only happen if either $2\abs{s}- 8k \ell = 0$ or $\sigma = 0$. In the former case, the left side of the equation becomes \[1 - A^{8k+8} - A^{8\ell + 8} + A^{8k + 8 \ell + 8} - A^{2 + 6k + 6 \ell - 12 k \ell} + A^{6 + 6k + 6 \ell- 12 k \ell} + A^{2 + 6k + 6 \ell+ 4 k \ell} - A^{6 + 6k + 6 \ell+ 4 k \ell}.\]

Here, the only exponents that could be equal to 8 are $6 + 6k + 6 \ell- 12 k \ell$ and $2 + 6k + 6 \ell - 12 k \ell$. Because it has 6 as a factor, $6 + 6k + 6 \ell- 12 k \ell$ cannot be equal to 8. If $2 + 6k + 6 \ell - 12 k \ell = 8$, we can write $k = \frac{\ell -1}{2\ell - 1}$. We see that $k$ is an increasing function of $\ell$ that limits to $1/2$ as $\ell$ approaches infinity. Since $k$ and $\ell$ are both integers that are at least 1, there are no solutions to this equation.

Since none of the terms can have an exponent equal to 8, it must be that $\sigma = 0$, so that $4\abs{s} = 2 + 6k + 6\ell + 4k \ell$. Then the left side of the equation becomes 
\[- A^{8k+8} - A^{8\ell + 8} + A^{8k + 8 \ell + 8} + A^{ 4} + A^{2 + 6k + 6\ell + 4k \ell} - A^{6 + 6k + 6\ell + 4k \ell} .\]

Because $A^4$ appears with a coefficient of $+1$ on the left of the equation, we must have that $2\abs{s}  =  8k\ell +4$. Using the two different expressions for $4\abs{s}$, we again have the equation $k = \frac{\ell -1}{2\ell - 1}$, which has no integral solutions.

We have found that no examples of twisted torus knots with trivial Jones polynomial arise from the case $p - 2k = q- 2\ell = 1$.

\textit{Case \ref{case:q1pbig}}: Suppose that $1 =  q-2\ell < p-2k$ and $\ell \ge 1$. From this equation $2\ell \equiv -1 \pmod q$. Let $i$ be that introduced in Lemma \ref{lem:k}. 

From Lemma \ref{lem:lisphat}, when $i = 0$, $\ell = \hat{p}^{-1}$, so we have $\hat{p} \equiv {-2} \pmod q$. Then $p  = m (2\ell+1) -2$ for some integer $m \ge 1$. Note here that because $p\ge 3$, $(\ell,m) \neq (1,1)$. From Corollary \ref{cor:ell}, $\ell (m(2\ell +1) - 2) + 0 - 1 = k (2\ell + 1)$. We can rearrange this equation to see that $(\ell m -k - 1) (2 \ell + 1)= 0$. Since $2 \ell + 1$ is not zero, we have that $k = \ell m - 1$. We use these new expressions of the parameters to write the auxiliary polynomial as:
\[\frac{ A^{4\ell(2\ell m+m -3)-2\abs{s}}}{1-A^{8}} \left[ X'_{2\ell m+m -2,2\ell + 1} - \sum_{i=0}^{\abs{s}-1}(-1)^i A^{4i} A^{-4-4\abs{s} + 8\ell m   + 4 \ell^2 m   }X'_{m,1} \right]. \] 

Because $X'_{m, 1} = 1 - A^{8}$, we use Lemma \ref{lem:smallcancellation} to say the auxiliary polynomial is equal to 1 exactly when \[1 - A^{8\ell m +4m -4} - A^{8\ell+8} + A^{8\ell m + 4m + 8\ell - 4} - A^{-4-4\abs{s} + 8\ell m   + 4 \ell^2 m } \hspace{2in} \]\[ \hspace{1in} + A^{-4\abs{s} + 8\ell m   + 4 \ell^2 m }  + A^{-4 + 8\ell m   + 4 \ell^2 m }  - A^{ 8\ell m   + 4 \ell^2 m } = A^{2\abs{s}-4\ell(2\ell m+m -3)}(1-A^{8}).\] 
Notice that because $\ell \ge 1$, $m \ge 1$ and $(\ell,m) \neq (1,1)$, the only exponents that could possibly be 0 on the left are those in which $\abs{s}$ appears. Further, in order for this equation to hold, either $2\abs{s}-4\ell(2\ell m+m -3) = 0$ or $-4-4\abs{s} + 8\ell m   + 4 \ell^2 m = 0$. In the former case, $\abs{s} = 2\ell (2\ell m+ m -3)$ and in the latter, $\abs{s} = \ell^2 m + 2\ell m -1$.

When $\abs{s} = 2\ell (2\ell m+ m -3)$, the equation is \[1 - A^{8\ell m +4m -4} - A^{8\ell+8} + A^{8\ell m + 4m + 8\ell - 4} - A^{-4-12\ell^2 m +24\ell} \hspace{.7in} \]\[ \hspace{1in}+ A^{- 12 \ell^2 m  +24 \ell }  + A^{-4 + 8\ell m   + 4 \ell^2 m }   - A^{ 8\ell m   + 4 \ell^2 m } = 1-A^8.\] Since $-A^8$ appears on the right, it must be the case that $-4-12\ell^2 m + 24 \ell = 8$ because the other terms with $-1$ as a coefficient have exponents that are at least 16. Then $\ell(2 - \ell m) = 1$, which means $(\ell,m) = (1,1)$, a contradiction.

When $\abs{s} = \ell^2 m + 2\ell m -1$, the equation becomes  \[ - A^{8\ell m +4m -4} - A^{8\ell+8} + A^{8\ell m + 4m + 8\ell - 4} + A^{4} + A^{-4 + 8\ell m   + 4 \ell^2 m } - A^{ 8\ell m   + 4 \ell^2 m } = A^{4 - 6 \ell^{2} m + 12 \ell}(1-A^{8}).\] Since $A^{4-6\ell^2 m +12 \ell}$ appears on the right, the exponent must be equal to $8 \ell m + 4m + 8 \ell -4$, $4$ or $ -4+ 8 \ell m + 4 \ell^2 m$. Additionally, $8-6\ell^2 m +12 \ell$ must be in the set $\{8 \ell m + 4m - 4, 8 \ell + 8, 8 \ell  + 4 \ell^2 m \}$. By analyzing each case, we find that $4- 6\ell^2 m + 12\ell = 8 \ell m + 4m + 8 \ell -4$ and $8-6\ell^2 m +12 \ell = 8 \ell m + 4m - 4$ are the only cases that do not, when considered individually, provide a contradiction. However, in conjunction, we find that these equations require that $4 = -8\ell$, a contradiction. 

Now we have exhausted the $i = 0$ case, so it must be that $i = 2$.

When $i = 2$, $\ell \equiv - \hat p^{-1} \pmod q$, so $2 \ell \equiv - 1 \pmod q$ tells us that $p \equiv 2 \pmod q$. That is $p = m(2\ell +1) + 2$ for some integer $m \ge 1$. (Here $m$ and $\ell$ can both be 1 without violating any restrictions on $p$.) Similarly to above, we find that $k = \ell m + 1$, and we rewrite the auxliary polynomial as \[\frac{ A^{4\ell(2\ell m+m +1)-2\abs{s}}}{1-A^{8}} \left[ X'_{2\ell m + m + 2,2 \ell + 1} - \sum_{i=0}^{\abs{s}-1}(-1)^i A^{4i} A^{8-4\abs{s} + 8 \ell + 8\ell m   + 4 \ell^2 m  }X'_{m,1} \right] \] 

Because $X'_{m, 1} = 1 - A^{8}$, Lemma \ref{lem:smallcancellation} tells us the auxiliary polynomial is 1 exactly when \[1 - A^{8\ell m +4m + 12} - A^{8\ell+8} + A^{8\ell m + 4m + 8\ell + 12} - A^{8-4\abs{s} + 8 \ell + 8\ell m   + 4 \ell^2 m} \hspace{2in} \]\[\hspace{.2in}+ A ^{12-4\abs{s} + 8 \ell + 8\ell m   + 4 \ell^2 m} + A^{8 + 8 \ell + 8\ell m   + 4 \ell^2 m} - A^{12 + 8 \ell + 8\ell m   + 4 \ell^2 m} = A^{2\abs{s} - 4 \ell(2\ell m + m + 1)}(1 - A^{8}).\] As in the previous cases, analyzing terms reveals we must have $2\abs{s} - 4 \ell (2 \ell m + m + 1) = 0$ or $8 - 4 \abs{s}+8 \ell + 8 \ell m + 4 \ell^2 m = 0$. In the former cases $\abs{s} = 2\ell (2 \ell m + m + 1)$ and in the latter $\abs{s} = \ell^2 m + 2\ell + 2\ell m + 2$.

When $\abs{s} = 2\ell (2 \ell m + m + 1)$ , our equation becomes \[1 - A^{8\ell m +4m + 12} - A^{8\ell+8} + A^{8\ell m + 4m + 8\ell + 12} - A^{8-12 \ell^2 m } \hspace{1in} \]\[\hspace{0.1in}+ A ^{12-12 \ell^2 m}+ A^{8 + 8 \ell + 8\ell m   + 4 \ell^2 m} - A^{12 + 8 \ell + 8\ell m   + 4 \ell^2 m} = 1 - A^{8}.\] Since none of the exponents on the left side of the equation can be equal to 8, this case is impossible. 

When $\abs{s} = \ell^2 m + 2\ell + 2\ell m + 2$, we have 
\[ - A^{8\ell m +4m + 12} - A^{8\ell+8} + A^{8\ell m + 4m + 8\ell + 12} + A ^{4} \hspace{1.5in} \]\[\hspace{.5in}+ A^{8 + 8 \ell + 8\ell m   + 4 \ell^2 m} - A^{12 + 8 \ell + 8\ell m   + 4 \ell^2 m} = A^{ 4 - 6\ell^2 m }(1 - A^{8}).\] Now $4-6\ell^2 m$ must be in the set $\{8\ell m + 4m + 8 \ell + 12, 4, 8 + 8\ell + 8 \ell m + 4 \ell^2 m\}$ and $12 - 6 \ell^2 m$ must be in the set $\{8 \ell m + 4m + 12, 8 \ell + 8, 12 + 8 \ell + 8 \ell m + 4\ell^2 m\}$. Again, when considered individually, the only possibility is that $4 - 6 \ell^2 m = 8\ell m + 4m + 8 \ell + 12$ and $12 - 6 \ell^2 m = 8 \ell m + 4m + 12$, but together these equations require $8 \ell = -8$, a contradiction.

\textit{Case \ref{case:p1qbig}}: Finally, suppose that $1 = p-2k < q - 2\ell$, so $-2k \equiv 1 \pmod p$. Let $i$ be that introduced in Lemma \ref{lem:k}. 

From Lemma \ref{lem:kisqhat}, when $i = 2$, $q \equiv k^{-1} \pmod p$. Putting these two modular equations together, we have $q = mp -2 = 2km + m -2$ for some integer $m \ge 1$. Since we are assuming $q \ge 3$, we must have that $(k,m) \neq (1,1)$. Using Corollary \ref{cor:ell}, we have $\ell = km - 1$. Then we can express the auxiliary polynomial in terms of only $k$ and $m$ as
\[\frac{ A^{2(2k)(2km+m-3)-2\abs{s}}}{1-A^{8}} \left[ X'_{2k+1,2km+m-2} - \sum_{i=0}^{\abs{s}-1}(-1)^i A^{4i} A^{-4-4\abs{s} + 8km + 4 k^2 m}X'_{1,m} \right]. \]

The polynomial $X'_{1,m}$ is $1-A^8$, so we use Lemma \ref{lem:smallcancellation}, with $\sigma = -4- 4 \abs{s} + 8km + 4k^2m$ to simplify the polynomial so that it is trivial exactly when 
\[ 1- A^{8k+8} - A^{8km+4m-4} + A^{8k + 8km + 4m -4}- A^{\sigma} \hspace{2in} \]\[\hspace{.65in} + A^{\sigma + 4} + A^{-4 + 8km + 4k^2m} - A^{8km + 4k^2m}  = A^{2\abs{s} - 4k(2km+m-3)}(1-A^{8}).\] Because $k, m \ge 1$, this happens only if either $2\abs{s} - 4k(2km+m-3) = 0$ or $\sigma = 0$. 

If $2\abs{s} - 4k(2km+m-3) = 0$, $\abs{s} = 4k^2m+2km-6k$ and the right side of the equation is $ 1-A^8$. One of the terms on the left with a coefficient of $-1$ must have an exponent of 8, so  either $8km+4m - 4 = 8$ or $\sigma =8$. In the former case $m(2k+1) = 3$, which requires $k = m = 1$, contradicting our assumption that $(k,m) \neq (1,1)$. In the latter case, $\sigma = 8$ tells us that $\abs{s} = 2km + k^2 m -3$. Then we have two expressions for $\abs{s}$ that, when combined, required $k=m=1$, a contradiction.

Then it must be that $\sigma = -4- 4 \abs{s} + 8km + 4k^2m = 0$, so we have $\abs{s} = 2km + k^2 m -1$. Now the exponent on the right becomes $2\abs{s} - 8k^2 m - 4km + 12k =-2 - 6k^2 m  + 12k$. It must be, then, that $-2 - 6k^2 m  + 12k$ is in the set $\{0, 4,  8k+8km+4m - 4 \}$. If $-2 - 6k^2 m  + 12k = 0$, we find that 2 is divisible by 6, and if 
$-2 - 6k^2 m  + 12k = 4$, we that $k = m = 1$, both of which are contradictions. If $-2 - 6k^2 m  + 12k = 8k+8km+4m - 4 $, we see that $3k^2 m + 2m (2k+1) = 2k+1$. However, $3k^2 m + 2m (2k+1) > 2k +1$, so we have another contradiction.

Since the $i =2$ case gave us no examples of knots with trivial Jones polynomial, we move on the the $i=0$ case. 

When $i = 0$, $q \equiv -k^{-1} \pmod p$, so we have $q = mp + 2 = 2km + m + 2$, where $m \ge 1$. Using Corollary \ref{cor:ell}, we have $\ell = km + 1$, and the auxiliary polynomial can be written as

\[\frac{A^{2(2k)(2km+m+1)-2\abs{s}}}{1-A^{8}} \left[ X'_{2k+1,2km+m+2} - \sum_{i=0}^{\abs{s}-1}(-1)^i A^{4i} A^{8-4\abs{s}+8k + 8 km + 4k^2 m}X'_{1,m} \right]. \]

Again, $X'_{1,m} = 1-A^8$, so we have the cancellation described in Lemma \ref{lem:smallcancellation}, and the auxiliary polynomial is trivial exactly when 
\[ 1 - A^{8k+8} - A^{8km+4m+12} + A^{8k+8km+4m+12}  - A^{\sigma} \hspace{2in} \]\[\hspace{.2in} + A^{\sigma + 4}  + A^{8+8k+8km+4k^2m} - A^{12+8k+8km+4k^2m} = A^{2\abs{s} - 4k(2km+m+1)}(1-A^{8}) \]

where $\sigma = 8-4\abs{s}+8k + 8 km + 4k^2 m$.

Because $k, m \ge 1$, this equation can hold only if $2\abs{s} - 4k(2km+m+1) = 0$ or $\sigma = 0$. If $2\abs{s} - 8k^2 m - 4km - 4k = 0$, then we must have $\sigma = 8$, so $2\abs{s} = 8k^2 m + 4km + 4k =  4k + 4km + 2k^2 m$, which reduces to $6k^2m = 0$, a contradiction.

If instead, $\sigma =0 $, we have $\abs{s} = 2+2k + 2 km + k^2 m$, so the exponent on the right of the equation becomes $4  - 6 k^2m$. Then we must satisfy the equation 
\[  -A^{8k+8} - A^{8km+4m+12} + A^{8k+8km+4m+12}  + A^{4} \hspace{1.8in} \]\[\hspace{.2in}  + A^{8+8k+8km+4k^2m} - A^{12+8k+8km+4k^2m} = A^{4 - 6 k^2 m}(1-A^{8}) \]

This requires that $4-6k^2m \in \{ 8k+8km+4m+12, 4, 8+8k+8km+4k^2m\}$, which is impossible since $4 - 6k^2m \le -2$.

We have exhausted all possible cases, and the only cases which allow for a trivial Jones polynomial are those which occur in Theorem \ref{thm:LeesUnknots}, so our proposition is proven.

\end{proof}

Our next application lives in the more general area of twisted torus links.

\begin{prop}
Let $p,q$ and $s$ be integers with $p \ge 3, q\ge 1$, $\gcd(p,q) = 1$, and $s>0$. The Jones polynomial for $T((p,q),(2,s))$ is not trivial.
\end{prop}

\begin{proof}
Again we work with the auxiliary polynomial of the mirror image of our desired knot. From Lemma \ref{lem:AuxPolyPos}, the auxiliary polynomial for $T\mi((p,q),(2,s))$ is 
\[\frac{(-1)^{s} A^{2(p-1)(q-1)+2s}}{1-A^{8}} \left[ X'_{p,q} - (-1)^{s}\sum_{i=0}^{s-1}(-1)^i A^{-4i} A^{4s-2 + 4(k+\ell-k\ell) + 2(kq + \ell p)}X'_{p-2k,q-2\ell} \right]\]
with $X'_{p,q} = 1 - A^{4p+4} - A^{4q+4} + A^{4p+4q}$ and $X'_{p-2k,q-2\ell} = 1 - A^{4p +4 - 8k} - A^{4q+4 - 8\ell} + A^{4p+4q - 8k - 8\ell}$. Suppose, on the contrary, that the auxiliary polynomial is equal to 1. We must have  
\[(-1)^s X'_{p,q} - \sum_{i=0}^{s-1}(-1)^i A^{-4i} A^{4s-2 + 4(k+\ell-k\ell) + 2(kq + \ell p)}X'_{p-2k,q-2\ell} = (1-A^{8})A^{-2(p-1)(q-1)-2s} \]

Because $p \ge 3$, $q \ge 1$, and $s\ge 1$, $2(p-1)(q-1)+2s \ge 6$, so the exponent $-2(p-1)(q-1)-2s$, on the right side of the equation, is not zero. Unless $2(p-1)(q-1)+2s = 8$ (so that the right side of the equation becomes $A^{-8}-1$), there is no constant term on the right. 

When $2(p-1)(q-1)+2s \neq 8$, the $(-1)^s$ term in $(-1)^s X'_{p,q}$ must cancel with some other term on the left. That is, there is an $i$ so that $-4i + 4s - 2 + 4(k+\ell-k\ell) + 2(kq+\ell p) + \alpha = 0$, where $\alpha \in \{0, 4p+4 - 8k, 4q+4 - 8 \ell, 4p + 4q - 8k - 8 \ell\}$. Rearranging this equation, we have $4i = 4s - 2 + 4(k+\ell) + 2kq + 2\ell(p - 2k) + \alpha$. Using the restrictions on the parameters, the right side of this equation has a lower bound of $4s + 14$. However, $0\le i \le s - 1$, so the left side of the equation has an upper bound of $4s-4$. Therefore we have that the $(-1)^s$ term of $(-1)^s X'_{p,q}$ persists, and $2(p-1)(q-1)+2s = 8$.

When $s$ is even, the $1$ on the left side of the equation appears with coefficient 1, but on the right side, the constant term appears with coefficient $-1$. This means there must exist two terms  with exponent 0 on the left side, which we have established to be impossible. Thus $s$ is odd.

The possible triples $(p,q,s)$ that provide a solution to the equation $2(p-1)(q-1)+2s = 8$ are  $(3, 1, 4)$, $(3, 2, 2)$, $(4, 1, 4)$, and $(4, 2, 1)$. The only one of these with odd $s$ is $(4,2,1)$, which requires $\gcd(p,q) = 2$, a contradiction. Hence the auxiliary polynomial of $T\mi((p,q),(2,s))$ is never trivial, and therefore the Jones polynomial of $T((p,q),(2,s))$ is never trivial.
\end{proof}

In this proof, the requirement that $p$ and $q$ are coprime rules out the link $T((4, 2),(2,1))$ as a potential option for a link with trivial auxiliary polynomial. By drawing this link, it is easy to see that it is isotopic to the torus knot $T(2,5)$, so its Jones polynomial is known to be nontrivial.

\begin{cor}
Let $p,q$ and $s$ be integers with $p \ge 3, q\ge 1$, $\gcd(p,q) = 1$, and $s$ positive and even. Then $T((p,q),(2,s))$ is not the unknot.
\end{cor}

In this work, we use a twisted torus link of length two, $T((p_1, q_1),(p_2, q_2))$, with $p_2 = 2$ to produce our computations. To generalize this work, one could work with $p_2=3$; however, more care would need to be taken in calculating the bracket polynomial: as the number of strands being twisted increases, the potential outcomes of resolution changes, in particular such as the one done in Lemma \ref{lem:skeinone}, grows exponentially.

\printbibliography

\vspace{.4in}
\noindent Brandon Bavier\\
West Virginia University\\
brandon.bavier@mail.wvu.edu

\vspace{.2in}
\noindent Brandy Doleshal\\ 
Sam Houston State University\\
bdoleshal@shsu.edu

\end{document}